\documentclass[12pt,reqno]{amsart}
\usepackage{a4wide}

\numberwithin{equation}{section}

\newtheorem{theorem}{Theorem}[section]
\newtheorem{proposition}[theorem]{Proposition}
\newtheorem{corollary}[theorem]{Corollary}
\newtheorem{lemma}[theorem]{Lemma}

\theoremstyle{definition}

\theoremstyle{remark}
\newtheorem{remark}[theorem]{Remark}

\newcommand{\ep}{\varepsilon}

\newcommand{\R}{\mathbb R_+^N}
\newcommand{\BR}{\partial \mathbb R_+^N}
\newcommand{\lb}{\lambda}

\newcommand{\intr}{\int_{\mathbb R^N}}

\begin{document}
\title{Liouville Type Theorems for Two Mixed Boundary Value Problems with General Nonlinearities}
\author{Xiaohui  Yu$^{*}$}
\thanks{*The Center for China's Overseas Interests,
Shenzhen University, Shenzhen Guangdong, 518060, The People's
Republic of China(yuxiao\_211@163.com)} \maketitle

{\scriptsize  {\bf Abstract:} In this paper, we study the
nonexistence of positive solutions for the following two mixed
boundary value problems. The first problem is the mixed
nonlinear-Neumann boundary value problem
$$
\left\{
  \begin{array}{ll}
  \displaystyle
-\Delta u=f(u)    &{\rm in}\quad \R,    \\
\displaystyle
\\ \frac{\partial u}{\partial \nu}=g(u) &{\rm on}\quad \Gamma_1,\\
\displaystyle
\\ \frac{\partial u}{\partial \nu}=0 &{\rm on}\quad \Gamma_0
\end{array}
\right.
$$
and the second is the nonlinear-Dirichlet boundary value problem
$$
\left\{
  \begin{array}{ll}
  \displaystyle
-\Delta u=f(u)    &{\rm in}\quad \R,    \\
\displaystyle
\\ \frac{\partial u}{\partial \nu}=g(u) &{\rm on}\quad \Gamma_1,\\
\displaystyle
\\ u=0 &{\rm on}\quad \Gamma_0,
\end{array}
\right.
$$
where $\R=\{x\in \mathbb R^N:x_N>0\}$, $\Gamma_1=\{x\in \mathbb
R^N:x_N=0,x_1<0\}$ and $\Gamma_0=\{x\in \mathbb R^N:x_N=0,x_1>0\}$.
We will prove that these problems possess no positive solution under
some assumptions on the nonlinear terms. The main technique we use
is the moving plane method in an integral form.

{\bf keywords:}\;\;Liouville type theorem, Moving plane method, Maximum principle, Mixed boundary value.\\

%\end {abstract}
\vspace{3mm} } \maketitle

\section {\bf Introduction}
In this paper, we study the nonexistence of positive solutions for
the following two mixed boundary value problems with general
nonlinearities. The first one is the nonlinear-Neumann boundary
value problem
\begin{equation}\label{1.1}
 \left\{
  \begin{array}{ll}
  \displaystyle
-\Delta u=f(u)    &{\rm in}\quad \R,    \\
\displaystyle
\\ \frac{\partial u}{\partial \nu}=g(u) &{\rm on}\quad \Gamma_1,\\
\displaystyle
\\ \frac{\partial u}{\partial \nu}=0 &{\rm on}\quad \Gamma_0
\end{array}
\right.
\end{equation}
and the second one is the nonlinear-Dirichlet problem
\begin{equation}\label{1.2}
\left\{
  \begin{array}{ll}
  \displaystyle
-\Delta u=f(u)    &{\rm in}\quad \R,    \\
\displaystyle
\\ \frac{\partial u}{\partial \nu}=g(u) &{\rm on}\quad \Gamma_1,\\
\displaystyle
\\ u=0 &{\rm on}\quad \Gamma_0,
\end{array}
\right.
\end{equation}
where $\R=\{x\in \mathbb R^N:x_N>0\},N\geq 3$, $\Gamma_1=\{x\in
\mathbb R^N:x_N=0,x_1<0\}$ and $\Gamma_0=\{x\in \mathbb
R^N:x_N=0,x_1>0\}$. In the following, we assume that $f,g$ are
continuous functions.

Liouville type theorems are close related to the existence results
and prior estimates for elliptic equations. More precisely, after
blowing up, elliptic equation in bounded domain turns to be an
equation in $\mathbb R^N$ or $\R$. Using the respective Liouville
theorem, we get a contradiction, so the prior estimate is proved.
Then we can use the topological method to prove the existence
results. For more details, we refer to \cite{FLN}\cite{GS} and etc.

In the past few decades, there are plenty of works on the Liouville
type theorems for elliptic equations and elliptic systems. The first
result is \cite{GS1}, in which the authors studied the nonexistence
results for the following elliptic equation
\begin{equation}\label{1.3}
-\Delta u=u^p\quad {\rm in}\quad \mathbb R^N,\quad u\geq 0.
\end{equation}
The authors proved, among other things, that the only solution for
problem (\ref{1.3}) is $u\equiv 0$ provided $0<p<\frac{N+2}{N-2}$.
This result is optimal in the sense that for any $p\geq
\frac{N+2}{N-2}$, there are infinitely many positive solutions to
(\ref{1.3}). Thus the Sobolev exponent $ \frac {N+2}{N-2}$ is the
dividing exponent between existence and nonexistence of positive
solutions. Nonexistence result for problem in half space was also
obtained in \cite{GS}. Later, W.Chen and C.Li proved similar
results by using the moving plane method in \cite{CL}. By using the
moving plane method, the authors proved the solution is symmetric in
every direction and with respect every point, hence the solution
must be $u\equiv 0$. Recently, in an interesting paper \cite{DG},
L.Damascelli and F.Gladiali studied the nonexistence result of
positive solution for the following nonlinear problem with general
nonlinearity
\begin{equation}\label{1.4}
-\Delta u=f(u)\quad {\rm in}\quad \mathbb R^N,\quad u\geq 0.
\end{equation}
If $f$ is assumed to be increasing and subcritical, then the authors
proved the only solution for problem \eqref{1.4} is $u\equiv 0$. The
main tool they used is the method of moving planes. We note that $f$
is only assumed to be continuous in this paper. So we can not
conclude that the weak solutions of problem (\ref{1.4}) are of $C^2$
class. In developing the method of moving plane, the authors in
\cite{DG} used the technique based on integral inequalities, an idea
originally due to S.Terracini's work \cite{Te1} and \cite{Te2}.
After the work of \cite{DG}, a lot of works concern the Liouville
type theorems for elliptic equation with general nonlinearity, see
\cite{GL}\cite{GL1}\cite{GLZ}\cite{LZ}\cite{Yu}\cite{Yu1}\cite{Yu3}\cite{Yu4}.
For Liouville theorem on nonlinear elliptic systems, we refer to
\cite{FF}\cite{Mi}\cite{SZ1}\cite{SZ2}\cite{SZ3}\cite{So} and etc..

In this paper, we are concerned with the nonexistence of positive
solution for the two mixed boundary value problems. Mixed boundary
problems were widely studied in the past few decades. For example,
E.Colorado and I.Peral studied the existence results in \cite{CP}.
The concentration behaviors of singularly perturbed mixed boundary
problems were studied in \cite{AMMP}\cite{AMMP2}\cite{DY} and the
references therein. Nonexistence results for Neumann-Dirichlet mixed
boundary problem have also been obtained in the past, see \cite{BGP}
and \cite{DG}. However, up to our knowledge, it seems to no result
for nonlinear-Neumann and nonlinear-Dirichlet mixed boundary value
problems and this is the main purpose of this paper.

First, we need to define weak solution for problem \eqref{1.1} and problem \eqref{1.2}. Let $E=W_{loc}^{1,2}(\mathbb R^N_+)\cap C^0(\overline{\R})$, we say that $u$ is a weak solution of problem \eqref{1.1}, if $u\in E$ and satisfies
$$
\intr \nabla u\nabla \varphi\,dx=\intr f(u)\varphi\,dx+\int_{\Gamma_1}g(u)\varphi\,dx'
$$
for any $\varphi \in C_c^1(\overline{\R})$. Similarly, we denote by $W=\{\varphi \in C_c^1(\overline {\R}), supt\{\varphi\}\subset A \}$
with $A=\R\cup \Gamma_1$, then we say that $u$ is a weak solution of problem \eqref{1.2}, if $u\in E$ and satisfies
$$
\intr \nabla u\nabla \varphi\,dx=\intr f(u)\varphi\,dx+\int_{\Gamma_1}g(u)\varphi\,dx'
$$
for any $\varphi \in W$.

Our first
result concerns the nonexistence of positive solution for
nonlinear-Neumann mixed boundary value problem, i.e., problem \eqref{1.1}.
Our main result is the following
\begin{theorem}\label{t 1.1}
Let $u\in E$ be a bounded nonnegative solution of
problem \eqref{1.1}, where $f,g:[0,+\infty)\to [0,+\infty)$ are
continuous functions with the properties

(i) $f(t),g(t)$ are nondecreasing in $(0,+\infty)$.

(ii) $h(t)=\frac{f(t)}{t^{\frac{N+2}{N-2}}},k(t)=
\frac{g(t)}{t^{\frac N{N-2}}}$ are nonincreasing in $(0,+\infty)$.

Then  $u\equiv c$ with $f(c)=g(c)=0$.
\end{theorem}

Next, we study the mixed nonlinear-Dirichlet mixed boundary value
problem \eqref{1.2}. Our second conclusion is the following
\begin{theorem}\label{t 1.2}
Let $u\in E$ be a bounded nonnegative solution of
problem \eqref{1.2}, where $f,g:[0,+\infty)\to [0,+\infty)$ are
continuous functions with the properties

(i) $f(t),g(t)$ are nondecreasing in $(0,+\infty)$.

(ii) $h(t)=\frac{f(t)}{t^{\frac{N+2}{N-2}}},k(t)=
\frac{g(t)}{t^{\frac N{N-2}}}$ are nonincreasing in $(0,+\infty)$.

Then  $u\equiv 0$.
\end{theorem}

\begin{remark}
By the Doubling Lemma in \cite{PQS}, the boundedness assumptions in
Theorem \ref{1.1} and Theorem \ref{1.2} can be dropped.
\end{remark}

Here we should emphasize that we only need that the nonlinearities
are continuous in the above theorems, we don't need they are
Lipschitz continuous. Since no Lipschitz assumption was made on the
nonlinearities, the usual maximum principle does not work. We
resorted to some integral inequality, which was first introduced by
S.Terracini in \cite{Te1}\cite{Te2} and then was widely used in
\cite{DG}\cite{GL}\cite{GL1} and \cite{GLZ} respectively. By the
same sprit of this, X.Yu studied the nonlinear Liouville type
theorems for other equations in \cite{Yu}\cite{Yu1}\cite{Yu3}\cite{Yu4} and
\cite{Yu5}. The main idea of this paper is the same as the above
works, we use integral inequality to substitute the usual maximum
principle.

The rest of this paper is devoted the proof of Theorem \ref{t 1.1}
and Theorem \ref{t 1.2}. We divide the proof of these theorems into
the following steps. First, we show that the nonnegative solutions of
problems \eqref{1.1} and \eqref{1.2} are nondecreasing as $x_1$
decreases. Next, we show that the nonnegative solutions either depend
only on $x_1$ and $x_N$ or are regular at infinity. Finally, if $u$
depends only on $x_1$ and $x_N$, then the limit function $w(x_N)=\lim_{x_1\to
-\infty}u(x_1,x_N)$ exists and $w(x_N)$ satisfies a proper equation
and it can only be the trivial solution under the assumptions in
Theorem \ref{t 1.1} and Theorem \ref{t 1.2}. On the other hand, if
the solution is regular at infinity, the we get a contradiction
directly from the monotonicity of $u$ in $x_1$ direction. In the
rest of this paper, we denote $C$ by a positive constant, which may
vary from line to line.

\section{\bf Proof of Theorem \ref{t 1.1}}
The key step is to prove the monotonicity of the nonnegative solutions
for problem \eqref{1.1} in the $x_1$ direction. We use the moving
plane method to prove our result. The first step of moving plane is
to show that this procedure can be started at some point. Since we
don't know the decay behaviors of $u$, it seems difficult to use
this method directly on $u$. So we make use of the Kelvin
transformation $v$ of $u$. More precisely, for any $\mu \in \mathbb
R$, we denote $p_\mu=(\mu,0,...,0)\in \BR$ and define the Kelvin
transformation $v_\mu$ of $u$ at $p_\mu$ as
\begin{equation}\label{2.1}
v_\mu(x)=\frac
1{|x-p_\mu|^{N-2}}u(\frac{x-p_\mu}{|x-p_\mu|^2}+p_\mu).
\end{equation}
Then we deduce from the definition of Kelvin transformation that
$v_\mu(x)$ decays at the rate of $|x-p_\mu|^{2-N}$ as $|x|\to
\infty$. In particular, we have
\begin{equation}\label{2.2}
v_\mu(x)\in L^{2^*}\cap L^\infty(\R\setminus B_r(p_\mu))
\end{equation}
for any $r>0$, where $2^*=\frac{2N}{N-2}$ is the usual Sobolev
critical exponent. Moreover, a direct calculation shows that $v_\mu$
satisfies the following equation
\begin{equation}\label{2.3}
 \left\{
\begin{array}{ll}
-\Delta v_\mu=\frac{f(|x-p_\mu|^{N-2}v_\mu(x))}{[|x-p_\mu|^{N-2}v_\mu(x)]^{\frac{N+2}{N-2}}}v_\mu(x)^{\frac{N+2}{N-2}}, & \ x\in \R,\\
\\ \frac{\partial v_\mu}{\partial \nu}=\frac{g(|x-p_\mu|^{N-2}v_\mu(x))}{[|x-p_\mu|^{N-2}v_\mu(x)]^{\frac{N}{N-2}}}v_\mu(x)^{\frac{N}{N-2}}, & \ x\in \{x\in \partial \R:\mu-\frac
1\mu<x_1<\mu\},\\
\\ \frac{\partial v_\mu}{\partial \nu}=0,& \ x\in \{x\in \partial \R:x_1>\mu\ {\rm or}\ x_1<\mu-\frac
1\mu\}
\end{array}
\right.
\end{equation}
for $\mu\geq 0$, where we interpret $\frac 1\mu=+\infty$ for
$\mu=0$. However, for $\mu<0$, $v_\mu$ satisfies
\begin{equation}\label{2.4}
 \left\{
\begin{array}{ll}
-\Delta v_\mu=\frac{f(|x-p_\mu|^{N-2}v_\mu(x))}{[|x-p_\mu|^{N-2}v_\mu(x)]^{\frac{N+2}{N-2}}}v_\mu(x)^{\frac{N+2}{N-2}}, & \ x\in \R,\\
\\ \frac{\partial v_\mu}{\partial \nu}=\frac{g(|x-p_\mu|^{N-2}v_\mu(x))}{[|x-p_\mu|^{N-2}v_\mu(x)]^{\frac{N}{N-2}}}v_\mu(x)^{\frac{N}{N-2}}, & \ x\in\{x\in \partial \R:x_1<\mu\ {\rm or}\ x_1>\mu-\frac
1\mu\}, \\
\\ \frac{\partial v_\mu}{\partial \nu}=0, & \ x\in \{x\in \partial \R:\mu<x_1<\mu-\frac
1\mu\}.
\end{array}
\right.
\end{equation}

Moreover, we infer from the definitions of $h$ and $k$ that $v_\mu$
satisfies
\begin{equation}\label{2.5}
 \left\{
\begin{array}{ll}
-\Delta v_\mu=h(|x-p_\mu|^{N-2}v_\mu(x))v_\mu(x)^{\frac{N+2}{N-2}}, & \ x\in \R,\\
\\ \frac{\partial v_\mu}{\partial \nu}=k(|x-p_\mu|^{N-2}v_\mu(x))v_\mu(x)^{\frac{N}{N-2}}, & \ x\in \{x\in \partial \R:\mu-\frac
1\mu<x_1<\mu\},\\
\\ \frac{\partial v_\mu}{\partial \nu}=0, & \ x\in \{x\in \partial \R:x_1>\mu\ {\rm or}\ x_1<\mu-\frac
1\mu\}
\end{array}
\right.
\end{equation}
for $\mu\geq 0$ and $v_\mu$ satisfies
\begin{equation}\label{2.6}
 \left\{
\begin{array}{ll}
-\Delta v_\mu=h(|x-p_\mu|^{N-2}v_\mu(x))v_\mu(x)^{\frac{N+2}{N-2}}, & \ x\in \R,\\
\\ \frac{\partial v_\mu}{\partial \nu}=k(|x-p_\mu|^{N-2}v_\mu(x))v_\mu(x)^{\frac{N}{N-2}}, & \ x\in\{x\in \partial \R:x_1<\mu\ {\rm or}\ x_1>\mu-\frac
1\mu\}, \\
\\ \frac{\partial v_\mu}{\partial \nu}=0,& \ x\in \{x\in \partial \R:\mu<x_1<\mu-\frac
1\mu\}
\end{array}
\right.
\end{equation}
for $\mu<0$.

We first study the case $\mu>0$. We will show that the moving plane
procedure can be carried out from $\infty$ to $\lb=\mu$. For this
purpose, we define
$$
\Sigma_\lb=\{x\in \R:x_1>\lb\},\ T_\lb=\{x\in \R:x_1=\lb\}
$$
and
$$
\partial \Sigma_\lb^1=\{x\in \partial\Sigma_\lb:x_N=0\}.
$$
Moreover, for any $x\in \Sigma_\lambda$, the reflection of $x$ with
respect to $T_\lambda$ is
$$
x^\lambda=\{2\lambda-x_1,x_2,...,x_N\}.
$$
In the following, we denote by $u^\lambda(x)=u(x^\lambda)$ and
$p_\mu^\lambda=\{2\lambda-\mu, 0,...,0\}$. With the above notations,
we have the following key lemma.
\begin{lemma}\label{t 2.1}
For any fixed $\lambda>\mu\geq 0$, the functions $v_\mu$ and
$(v_\mu-v_\mu^\lambda)^+$ belong to $ L^{2^*}(\Sigma_\lambda)\cap
L^\infty(\Sigma_\lambda)$ with $2^*=\frac{2N}{N-2}$. Further more,
if we denote $A_\mu^\lambda=\{x\in \Sigma_\lambda|
v_\mu>v_\mu^\lambda\}$ and $B_\mu^\lambda=\{x\in \partial
\Sigma_\lambda^1|v_\mu(x)>v_\mu^\lambda(x)\}$, then there exists
$C_\lambda>0$, which is nonincreasing in $\lambda$, such that
\begin{equation}\label{2.7}
\int_{\Sigma_\lambda}|\nabla (v_\mu-v_\mu^\lambda)^+|^2\,dx\leq
C_\lambda(\int_{A_\mu^\lambda}\frac
1{|x-p_\mu|^{2N}}\,dx)^{\frac2N}\cdot(\int_{\Sigma_\lambda}|\nabla(v_\mu-v_\mu^\lambda)^+|^2\,dx).
\end{equation}
\end{lemma}

\begin{proof}
Since $\lambda>\mu$, there exists $r>0$ such that
$\Sigma_\lambda\subset \R\setminus B_r(p_\mu)$. Hence, we deduce
from equation \eqref{2.2} and the decay behavior of $v_\mu$ that
$$
v_\mu,(v_\mu-v_\mu^\lambda)^+\in L^{2^*}(\Sigma_\lambda)\cap
L^\infty(\Sigma_\lambda).
$$
Now we choose a cut off function $\eta=\eta_\ep\in C(\overline \R, [0,1])$ such that
 $$
\eta(x)= \left\{
\begin{array}{ll}
1 & \quad {\rm for}\ 2\ep\leq |x-p_\mu^\lambda|\leq \frac 1\ep,\\
 0 & \quad {\rm for}\ |x-p_\mu^\lambda|<\ep\ {\rm or}\ |x-p_\mu^\lambda|> \frac 2\ep,
\end{array}
\right.
$$
$|\nabla \eta|\leq \frac 2\ep $ for $\ep<|x-p_\mu^\lambda|<2\ep$ and
$|\nabla \eta|\leq 2\ep $ for $\frac 1\ep<|x-p_\mu^\lambda|<\frac
2\ep$. Moreover, if we define
$\varphi=\varphi_\ep=\eta^2_\ep(v_\mu-v_\mu^\lambda)^+$ and
$\psi=\psi_\ep=\eta_\ep(v_\mu-v_\mu^\lambda)^+$, then a direct
calculation shows that
$$
|\nabla \psi|^2=\nabla(v_\mu-v_\mu^\lambda)\nabla
\varphi+[(v_\mu-v_\mu^\lambda)^+]^2|\nabla \eta|^2.
$$
On the other hand, it is easy to see that equation
\begin{equation}\label{2.8}
-\Delta(v_\mu(x)-v_\mu^\lambda(x))=h(|x-p_\mu|^{N-2}v_\mu(x))v_\mu(x)^{\frac{N+2}{N-2}}-h(|x^\lb-p_\mu|^{N-2}v_\mu(x^\lb))v_\mu(x^\lb)^{\frac{N+2}{N-2}}
\end{equation}
holds in $\Sigma_\lb\setminus \{p_\mu^\lambda\}$. So if we multiply
equation \eqref{2.8} by $\varphi$, then we get
\begin{equation}\label{2.9}
\begin{split}
&\int_{\Sigma_\lambda\cap\{2\varepsilon\leq |\xi-p^\lambda|\leq
\frac 1\ep\}}|\nabla(v_\mu-v_\mu^\lambda)^+|^2\,dx\\&\leq
 \int_{\Sigma_\lambda}|\nabla \psi|^2\,dx\\&=
 \int_{\Sigma_\lambda}\nabla(v_\mu-v_\mu^\lambda)\nabla\varphi\,dx+
 \int_{\Sigma_\lambda}[(v_\mu-v_\mu^\lambda)^+]^2|\nabla\eta_\ep|^2\,dx\\&= \int_{A_\mu^\lambda}-\Delta(v_\mu-v_\mu^\lambda)\varphi\,dx+
 \int_{\partial\Sigma_\lambda^1}\frac{\partial (v_\mu-v_\mu^\lambda)}{\partial \nu}\varphi\,dx'+I_\ep\\&\leq \int_{A_\mu^\lambda}
 [h(|x-p_\mu|^{N-2}v_\mu(x))v_\mu^{\frac{N+2}{N-2}}-h(|x^\lb-p_\mu|^{N-2}v_\mu^\lambda)(v_\mu^\lambda)^{\frac{N+2}{N-2}}]\varphi\,dx+I_\ep,
 \end{split}
\end{equation}
where
$I_\ep=\int_{\Sigma_\lambda}[(v-v^\lambda)^+]^2|\nabla\eta_\ep|^2\,dx$,
and the last inequality follows from that
$\frac{\partial(v-v^\lb)}{\partial \nu}\leq 0$ on $\partial
\Sigma_\lb^1$. Since $h$ is nonincreasing, the above equation
implies
\begin{equation}\label{2.10}
\begin{split}
&\int_{\Sigma_\lambda\cap\{2\varepsilon\leq |\xi-p^\lambda|\leq
\frac 1\ep\}}|\nabla(v_\mu-v_\mu^\lambda)^+|^2\,dx\\&\leq
 \int_{A_\mu^\lambda}
 h(|x-p_\mu|^{N-2}v_\mu(x))[v_\mu^{\frac{N+2}{N-2}}-(v_\mu^\lambda)^{\frac{N+2}{N-2}}]\varphi\,dx+I_\ep\\&\leq
 C\int_{A_\mu^\lambda}
 h(|x-p_\mu|^{N-2}v_\mu(x))v_\mu^{\frac{4}{N-2}}[v_\mu-v_\mu^\lambda]\varphi\,dx+I_\ep.
 \end{split}
\end{equation}
Moreover, since $\Sigma_{\lambda'}\subset \Sigma_\lambda$ for
$\lambda'>\lb$ and $|x-p_\mu|^{N-2}v_\mu(x)$ is bounded in
$\Sigma_\lb$, then there exists a constant $C_\lb$, which is
nonincreasing in $\lb$, such that
\begin{equation}\label{2.11}
\begin{split}
&\int_{\Sigma_\lambda\cap\{2\varepsilon\leq |\xi-p^\lambda|\leq
\frac 1\ep\}}|\nabla(v_\mu-v_\mu^\lambda)^+|^2\,dx\\&\leq
 C_\lb \int_{A_\mu^\lambda}
 v_\mu^{\frac{4}{N-2}}[v_\mu-v_\mu^\lambda]^+\varphi\,dx+I_\ep.
\end{split}
\end{equation}
On the other hand, we deduce from the decay property of $v_\mu$ that
\begin{equation}\label{2.12}
\begin{split}
&\int_{\Sigma_\lambda\cap\{2\varepsilon\leq |\xi-p^\lambda|\leq
\frac 1\ep\}}|\nabla(v_\mu-v_\mu^\lambda)^+|^2\,dx\\&\leq
 C_\lb \int_{A_\mu^\lambda}
  \frac 1{|x-p_\mu|^4}[(v_\mu-v_\mu^\lambda)^+]^2\eta^2\,dx+I_\ep\\&\leq C_\lb (\int_{A_\mu^\lambda}\frac 1{|x-p_\mu|^{2N}}dx)^{\frac 2N}
(\int_{\Sigma_\lambda}[(v_\mu-v_\mu^\lambda)^+]^{\frac{2N}{N-2}}dx)^{\frac{N-2}N}+I_\ep.
\end{split}
\end{equation}
We claim that $I_\ep\to 0$ as $\ep\to 0$. In fact, if we denote
$D_\ep^1=\{x\in \Sigma_\lambda|\ep<|x-p_\mu^\lambda|<2\ep\}$ and
$D_\ep^2=\{x\in \Sigma_\lambda|\frac 1\ep<|x-p_\mu^\lambda|<\frac
2\ep\}$, then we get
$$
\int_{D_\ep^1}|\nabla \eta|^N\,dx\leq C\frac 1{\ep^N}\cdot \ep^N=C.
$$
Similarly, we also have
$$
\int_{D_\ep^2}|\nabla \eta|^N\,dx\leq C\ep^N\cdot \frac 1{\ep^N}=C.
$$
Hence, it follows from the Holder inequality that
$$
I_\ep\leq (\int_{D_\ep^1\cup
D_\ep^2}[(v_\mu-v_\mu^\lambda)^+]^{2^*}\,dx)^{\frac 2{2^*}}\cdot
(\int_{\Sigma}|\nabla \eta|^N\,dx)^{\frac 2N} \to 0
$$
as $\ep \to 0$ since $(v-v^\lambda)^+\in L^{2^*}(\Sigma_\lambda)$.

Finally, letting $\ep\to 0$ in equation \eqref{2.12} and using the
dominated convergence theorem, Sobolev inequality, we get
\begin{equation}\label{2.13}
\int_{\Sigma_\lb}|\nabla(v_\mu-v_\mu^\lambda)^+|^2\,dx \leq
 C_\lb (\int_{A_\mu^\lambda}\frac 1{|x-p_\mu|^{2N}}dx)^{\frac 2N}
\int_{\Sigma_\lambda}|\nabla(v_\mu-v_\mu^\lambda)^+|^2\,dx,
\end{equation}
which completes the proof of this lemma.
\end{proof}

Before we continue the proof of Theorem \ref{t 1.1}, we give some
comments on Lemma \ref{t 2.1}. Since we don't require the nonlinear
terms $f$ and $g$ are Lipschitz continuous, the usual maximum
principle does not work. Thanks to equation \eqref{2.7}, it can play
the same role as the maximum principle. In fact, if we can prove
$$C_\lb (\int_{A_\mu^\lambda}\frac 1{|x-p_\mu|^{2N}}dx)^{\frac 2N}<1,$$
 then we get
$v_\mu\leq v_\mu^\lambda$, the same conclusion as the maximum
principle implies.

The next lemma shows that we can start the moving plane from some
place.
\begin{lemma}\label{t 2.2}
Under assumptions of Theorem \ref{t 1.1}, there exists
$\lambda_0>\mu$, such that for all $\lambda\geq \lambda_0$, we have
$v_\mu\leq v_\mu^\lambda$ in $\Sigma_\lambda$.
\end{lemma}
\begin{proof}
The conclusion of this lemma is a direct corollary of Lemma \ref{t
2.1}. In fact, by the decay speed of $v_\mu$, we can choose
$\lambda_0$ large enough such that
$$C_\lb (\int_{\Sigma_\lambda}\frac 1{|x-p_\mu|^{2N}}dx)^{\frac 2N}<\frac 12$$ for all $\lambda\geq \lambda_0$, then equation (\ref{2.7})
implies that
$$\int_{\Sigma_\lambda}|\nabla(v_\mu-v_\mu^\lambda)^+|^2\,dx=0.$$
The assertion follows.
\end{proof}
Now we can move the plane from the right to the left such that
$v_\mu\leq v_\mu^\lb$ in $\Sigma_\lb$ and suppose this process stops
at some $\lb_1$. More precisely, we define
\begin{equation}\label{2.14}
\lambda_1=\inf\{\lambda|v_\mu\leq v_\mu^\lambda\ {\rm in}\
\Sigma_\lambda \},
\end{equation}
then we have the following
\begin{lemma}\label{t 2.3}
If $u\not\equiv 0$, then $\lambda_1\leq \mu$.
\end{lemma}
\begin{proof}
We prove the conclusion by contradiction. Suppose on the contrary
that $\lb_1>\mu$, then we claim that $v_\mu\equiv
v_\mu^{\lambda_1}$.

We prove the claim by contradiction. Suppose $v_\mu\not\equiv
v_\mu^{\lambda_1}$, then we show that the plane can be moved to the
left a little. That is, we will show that there exists $\delta>0$,
such that $v_\mu(x)\leq v_\mu^\lambda(x)$ in $\Sigma_\lambda$ for
all $\lambda\in [\lambda_1-\delta,\lambda_1]$. This contradicts the
choice of $\lambda_1$.

To prove this claim, we first infer from the continuity that
$v_\mu(x)\leq v_\mu^{\lambda_1} (x)$. Moreover, we have
\begin{eqnarray*}
h(|x-p_\mu|^{N-2}v_\mu)v_\mu^{\frac{N+2}{N-2}}&=&\frac{f(|x-p_\mu|^{N-2}v_\mu)}{|x-p_\mu|^{N+2}}
\\&\leq&\frac{f(|x-p_\mu|^{N-2}v_\mu^{\lambda_1})}{|x-p_\mu|^{N+2}}
\\&=&\frac{f(|x-p_\mu|^{N-2}v_\mu^{\lambda_1})}{[|x-p_\mu|^{N-2}
v_\mu^{\lambda_1}]^{\frac{N+2}{N-2}}}(v_\mu^{\lambda_1})^{\frac{N+2}{N-2}}
\\&\leq&\frac{f(|x^{\lambda_1}-p_\mu|^{N-2}v_\mu^{\lambda_1})}{[|x^{\lambda_1}-p_\mu|^{N-2}
v_\mu^{\lambda_1}]^{\frac{N+2}{N-2}}}(v_\mu^{\lambda_1})^{\frac{N+2}{N-2}}
\\&=&h(|x^{\lambda_1}-p_\mu|^{N-2}v_\mu^{\lambda_1})(v_\mu^{\lambda_1})^{\frac{N+2}{N-2}},
\end{eqnarray*}
where the first inequality holds since $f$ is nondecreasing, the
second inequality is a consequence of (ii) in Theorem \ref{t 1.1}.

 This equation implies
$$
-\Delta v_\mu\leq -\Delta v_\mu^{\lambda_1},
$$
then we infer from the maximum principle that
$v_\mu<v_\mu^{\lambda_1}$ in $\Sigma_{\lambda_1}$ since
$v_\mu\not\equiv v_\mu^{\lambda_1}$. Moreover, since $\frac
1{|x-p_\mu|^{2N}}\chi_{ A_\mu^\lambda}\to 0$ a.e. as
$\lambda\to\lambda_1$ and $\frac
1{|x-p_\mu|^{2N}}\chi_{A_\mu^\lambda}\leq \frac
1{|x-p_\mu|^{2N}}\chi_{A_\mu^{\lambda_1-\delta}}$ for $\lambda\in
[\lambda_1-\delta, \lambda_1]$ and some $\delta>0$, then the
dominated convergence theorem implies
$$
\int_{A_\mu^{\lambda}}\frac 1{|x-p_\mu|^{2N}}\,dx\to 0
$$
as $\lambda\to \lambda_1$. That is, there exists $\delta>0$, such
that
$$
C_\lb (\int_{A_\mu^\lambda}\frac 1{|x-p_\mu|^{2N}}dx)^{\frac
2N}<\frac 12
$$
for all $\lambda\in[\lambda_1-\delta,\lambda_1]$. Then we infer from
Lemma \ref{t 2.1} that $v_\mu\leq v_\mu^\lambda$ for all
$\lambda\in[\lambda_1-\delta,\lambda_1]$. This contradicts the
definition of $\lambda_1$. So we prove the claim.

Next, we show that $v_\mu\equiv v_\mu^{\lb_1}$ implies $v_\mu\equiv
0$ and hence $u\equiv 0$. In fact, if $v_\mu\equiv
v_\mu^{\lambda_1}$, then we deduce from the Neumann boundary value
condition and the nonlinear boundary value condition that
$v_\mu(x)=0$ for some $x\in \BR$, this contradicts the Hopf Theorem
unless $v_\mu\equiv 0$.
\end{proof}

The above two lemmas implies the following
\begin{corollary}\label{t 2.22}
For any $\mu \geq 0$, we have $v_\mu(x)\leq v_\mu(x^\mu)$ for $x\in \Sigma_{\mu}$.
\end{corollary}

Next, we consider the case $\mu<0$. In this case, it seems difficult to start the moving plane method as $\mu>0$ because of the boundary
value condition of $v_\mu$. Instead of choosing $\mu$ arbitrary for $\mu\geq 0$, we move $\mu$ from $0$ to $-\infty$ gradually. We first need some integral
inequality similar to equation \eqref{2.7} to substitute the maximum
principle. In this case, we need to evaluate the term
$\int_{\partial \Sigma_\lb^1}\frac{\partial(v-v^\lb)}{\partial
\nu}\varphi\,dx'$ in equation \eqref{2.9} more carefully since it is
not nonpositive any longer. Similar to Lemma \ref{t 2.1}, we have
the following result.
\begin{lemma}\label{t 2.4}
Suppose that $\mu<0$, then for any fixed $\mu<\lambda\leq \mu-\frac
1{2\mu}$, the functions $v_\mu$ and $(v_\mu-v_\mu^\lambda)^+$ belong
to $ L^{2^*}(\Sigma_\lambda)\cap L^\infty(\Sigma_\lambda)$ with
$2^*=\frac{2N}{N-2}$. Further more, if we denote
$A_\mu^\lambda=\{x\in \Sigma_\lambda| v_\mu>v_\mu^\lambda\}$ and
$B_\mu^\lambda=\{x\in
\partial \Sigma_\lambda^1|v_\mu(x)>v_\mu^\lambda(x)\}$, then there exists
$C_\lambda>0$, which is nonincreasing in $\lambda$, such that
\begin{equation}\label{2.15}
\begin{split}
&\int_{\Sigma_\lambda}|\nabla (v_\mu-v_\mu^\lambda)^+|^2\,dx\\&\leq
C_\lambda[(\int_{A_\mu^\lambda}\frac
1{|x-p_\mu|^{2N}}\,dx)^{\frac2N}+(\int_{B_\mu^\lambda}\frac
1{|x-p_\mu|^{2(N-1)}}\,dx')^{\frac
1{N-1}}]\cdot(\int_{\Sigma_\lambda}|\nabla(v_\mu-v_\mu^\lambda)^+|^2\,dx).
\end{split}
\end{equation}
\end{lemma}
\begin{proof}
The proof of this lemma is similar to the proof of Lemma \ref{t
2.1}. We sketch it. We denote $\eta, \varphi,\psi$ as the proof of
Lemma \ref{t 2.1}, then we have the following inequality similar to
equation \eqref{2.9}. It reads
\begin{equation}\label{2.16}
\begin{split}
&\int_{\Sigma_\lambda\cap\{2\varepsilon\leq |\xi-p^\lambda|\leq
\frac 1\ep\}}|\nabla(v_\mu-v_\mu^\lambda)^+|^2\,dx\\&\leq
\int_{A_\mu^\lambda}-\Delta(v_\mu-v_\mu^\lambda)\varphi\,dx+
 \int_{\partial\Sigma_\lambda^1}\frac{\partial (v_\mu-v_\mu^\lambda)}{\partial
 \nu}\varphi\,dx'+I_\ep\\&\leq \int_{A_\mu^\lambda}
 [h(|x-p_\mu|^{N-2}v_\mu(x))v_\mu^{\frac{N+2}{N-2}}-h(|x^\lb-p_\mu|^{N-2}v_\mu^\lambda)(v_\mu^\lambda)^{\frac{N+2}{N-2}}]\varphi\,dx
 \\&+\int_{B_\mu^\lb}[k(|x-p_\mu|^{N-2}v_\mu(x))v_\mu^{\frac{N}{N-2}}-k(|x^\lb-p_\mu|^{N-2}v_\mu^\lambda)(v_\mu^\lambda)^{\frac{N}{N-2}}]\varphi\,dx'+I_\ep,
 \end{split}
\end{equation}
where
$I_\ep=\int_{\Sigma_\lambda}[(v_\mu-v_\mu^\lambda)^+]^2|\nabla\eta_\ep|^2\,dx$.
Since $h$ and $k$ are nonincreasing, the above equation implies
\begin{equation}\label{2.17}
\begin{split}
&\int_{\Sigma_\lambda\cap\{2\varepsilon\leq |\xi-p^\lambda|\leq
\frac 1\ep\}}|\nabla(v_\mu-v_\mu^\lambda)^+|^2\,dx\\&\leq
 \int_{A_\mu^\lambda}
 h(|x-p_\mu|^{N-2}v_\mu(x))[v_\mu^{\frac{N+2}{N-2}}-(v_\mu^\lambda)^{\frac{N+2}{N-2}}]\varphi\,dx\\&
 +\int_{B_\mu^\lb}k(|x-p_\mu|^{N-2}v_\mu(x))[v_\mu(x)^{\frac{N}{N-2}}-(v_\mu^\lambda)^{\frac{N}{N-2}}]\varphi\,dx'+I_\ep\\&\leq
 C\int_{A_\mu^\lambda}
 h(|x-p_\mu|^{N-2}v_\mu(x))v_\mu^{\frac{4}{N-2}}[v_\mu-v_\mu^\lambda]\varphi\,dx\\&
 +C\int_{B_\mu^\lb}k(|x-p_\mu|^{N-2}v_\mu(x))v_\mu(x)^{\frac{2}{N-2}}[v_\mu(x)-v_\mu(x^\lambda)]\varphi\,dx'+I_\ep.
 \\&\leq
 C_\lb \int_{A_\mu^\lambda}
 v_\mu^{\frac{4}{N-2}}[v_\mu-v_\mu^\lambda]^+\varphi\,dx
 +C_\lb\int_{B_\mu^\lb}v_\mu(x)^{\frac{2}{N-2}}[v_\mu(x)-v_\mu(x^\lambda)]\varphi\,dx'+I_\ep
 \\&\leq C_\lb (\int_{A_\mu^\lambda}\frac 1{|x-p_\mu|^{2N}}dx)^{\frac 2N}(\int_{\Sigma_\lambda}[(v_\mu-v_\mu^\lambda)^+]^{\frac{2N}{N-2}}dx)^{\frac{N-2}N}
 \\&+C_\lb(\int_{B_\mu^\lambda}\frac
1{|x-p_\mu|^{2(N-1)}}\,dx')^{\frac 1{N-1}}
(\int_{\partial\Sigma_\lambda^1}[(v_\mu-v_\mu^\lambda)^+]^{\frac{2(N-1)}{N-2}}\,dx')^{\frac{N-2}{N-1}}+I_\ep
\\&\leq
C_\lambda[(\int_{A_\mu^\lambda}\frac
1{|x-p_\mu|^{2N}}\,dx)^{\frac2N}+(\int_{B_\mu^\lambda}\frac
1{|x-p_\mu|^{2(N-1)}}\,dx')^{\frac
1{N-1}}]\cdot(\int_{\Sigma_\lambda}|\nabla(v_\mu-v_\mu^\lambda)^+|^2\,dx)+I_\ep.
 \end{split}
\end{equation}
Letting $\ep\to 0$ in the above equation and using the dominated
convergence theorem, Sobolev inequality and Sobolev trace
inequality, then we get equation \eqref{2.15}, which completes the
proof of this lemma.

\end{proof}

With the above preparations, we can start the moving plane procedure for $\mu<0$
now. However, in the case $\mu<0$, it seems difficult to move the plane
$T_\lb$ because of the boundary value conditions of $v_\mu$. Instead of choosing $\mu$ arbitrary for $\mu
\geq 0$, we move the plane $\mu$ from $0$ to $-\infty$ step by step. In the first step, we will show that the parameter
$\mu$ can be moved from $0$ to $\mu_1=-\frac{\sqrt 2}{2}$. This procedure is composed of the following two lemmas.
First, we show that this procedure can be
started for $\mu$ small enough. More precisely, we have the
following result.
\begin{lemma}\label{t 2.5}
There exists $\mu_0<0$, such that for all
$\mu_0\leq \mu\leq 0$, we have $v_\mu\leq v^0_\mu$ for all $x\in \Sigma_0$.
\end{lemma}
\begin{proof}
By Corollary \ref{t 2.22}, we know that $v_0\leq v_0^0$ for all $x\in \Sigma_0$, hence the functions
$\frac{1}{|x-p_\mu|^{2N}}\chi_{A_\mu^0}\to 0$ and
$\frac{1}{|x-p_\mu|^{2(N-1)}}\chi_{B_\mu^0} \to 0$ a.e. as $\mu\to
0^-$. We deduce from the dominate convergence theorem that
there exists $\mu_0<0$ small enough, such that for all $\mu_0\leq
\mu\leq 0$, the following inequality holds
$$
C_\lambda[(\int_{A_\mu^0}\frac
1{|x-p_\mu|^{2N}}\,dx)^{\frac2N}+(\int_{B_\mu^0}\frac
1{|x-p_\mu|^{2(N-1)}}\,dx')^{\frac 1{N-1}}]\leq \frac 12.
$$
Then we infer from Lemma \ref{t 2.4} that $\int_{A_\mu^0}|\nabla
(v_\mu-v_\mu^0)^+|\,dx=0$, that is $v_\mu\leq v_\mu^0$ in $\Sigma_0$.

\end{proof}

Next, we will show that the plane can be moved down
to $-\frac {\sqrt 2}{2}$. More precisely, we define
$$
\mu_1=\inf\{\mu<0:v_\mu\leq v_\mu^0\ {\rm in}\ \Sigma_0\},
$$
then we have
\begin{lemma}\label{t 2.6}
If $u\not\equiv 0$, then we must have $\mu_1\leq -\frac
{\sqrt 2}{2}$.
\end{lemma}
\begin{proof}
We prove the conclusion by contradiction. Suppose on the contrary
that $\mu_1>- \frac {\sqrt 2}{2}$, then we claim
that $v_{\mu_1}\equiv v_{\mu_1}^{0}$.

We prove the claim by contradiction. Suppose $v_{\mu_1}\not\equiv
v_{\mu_1}^{0}$, then we show that the plane can be moved to
the left a little. That is, we will show that there exists
$\delta>0$, such that $v_\mu(x)\leq v_\mu^0(x)$ in
$\Sigma_0$ for all $\mu\in [\mu_1-\delta,\mu_1]$. This
contradicts the choice of $\mu_1$.

To prove this, we first infer from the continuity that
$v_{\mu_1}(x)\leq v_{\mu_1}^{0} (x)$ for $x\in \Sigma_0$.
Moreover, we have
\begin{eqnarray*}
h(|x-p_{\mu_1}|^{N-2}v_{\mu_1})v_{\mu_1}^{\frac{N+2}{N-2}}&=&\frac{f(|x-p_{\mu_1}|^{N-2}v_{\mu_1})}{|x-p_{\mu_1}|^{N+2}}\\
&\leq&\frac{f(|x-p_{\mu_1}|^{N-2}v_{\mu_1}^{0})}{|x-p_{\mu_1}|^{N+2}}
\\&=&\frac{f(|x-p_{\mu_1}|^{N-2}v_{\mu_1}^{0})}{[|x-p_{\mu_1}|^{N-2}
v_{\mu_1}^{0}]^{\frac{N+2}{N-2}}}(v_{\mu_1}^{0})^{\frac{N+2}{N-2}}\\
&\leq&\frac{f(|x^{0}-p_{\mu_1}|^{N-2}v_{\mu_1}^{0})}{[|x^{0}-p_{\mu_1}|^{N-2}
v_{\mu_1}^{0}]^{\frac{N+2}{N-2}}}(v_{\mu_1}^{0})^{\frac{N+2}{N-2}}\\
&=&h(|x^{0}-p_{\mu_1}|^{N-2}v_{\mu_1}^{0})(v_{\mu_1}^{0})^{\frac{N+2}{N-2}},
\end{eqnarray*}
where the first inequality holds since $f$ is nondecreasing, the
second inequality is a consequence of (ii) in Theorem \ref{t 1.1}.

 This equation implies
$$
-\Delta v_{\mu_1}\leq -\Delta v_{\mu_1}^{0},
$$
then we infer from the maximum principle that
$v_{\mu_1}<v^{0}_{\mu_1}$ in $\Sigma_{0}$ since
$v_{\mu_1}\not\equiv v_{\mu_1}^{0}$. Since $\frac
1{|x-p_\mu|^{2N}}\chi_{ A_\mu^0}\to 0,\ \frac
1{|x-p_\mu|^{2(N-1)}}\chi_{B_\mu^0}\to 0$ a.e. as
$\mu\to\mu_1$, moreover, as $\frac
1{|x-p_\mu|^{2N}}\chi_{A_\mu^0}\leq \frac
1{|x-p_\mu|^{2N}}\chi_{A_{\mu_1-\delta}^{0}}$ and $\frac
1{|x-p_\mu|^{2(N-1)}}\chi_{B_\mu^0}\leq \frac
1{|x-p_\mu|^{2(N-1)}}\chi_{B_{\mu_1-\delta}^{0}} $ for $\mu\in
[\mu_1-\delta, \mu_1]$ and some $\delta>0$, then the dominated
convergence theorem implies
$$
\int_{A_\mu^{0}}\frac
1{|x-p_\mu|^{2N}}\,dx+\int_{B_\mu^0}\frac
1{|x-p_\mu|^{2(N-1)}}\,dx'\to 0
$$
as $\mu\to \mu_1$. That is, there exists $\delta>0$, such that
$$
C_\lb[ (\int_{A_\mu^{0}}\frac 1{|x-p_\mu|^{2N}}\,dx)^{\frac
2N}+(\int_{B_\mu^0}\frac 1{|x-p_\mu|^{2(N-1)}}\,dx')^{\frac
1{N-1}}]<\frac 12
$$
for all $\mu\in[\mu_1-\delta,\mu_1]$. Then we infer from Lemma
\ref{t 2.4} that $v_\mu\leq v_\mu^0$ for all
$\mu\in[\mu_1-\delta,\mu_1]$. This contradicts the definition of
$\mu_1$. This proves the claim.

Similar to Lemma \ref{t 2.3}, we can show that $v_{\mu_1}(x)\equiv
v_{\mu_1}^0$ implies $v_{\mu_1}\equiv 0$. In fact, if
$v_{\mu_1}\equiv v_{\mu_1}^{0}$, then we deduce from the
Neumann boundary value condition and the nonlinear boundary value
condition that $v_{\mu_1}(x)=0$ for some $x\in \BR$, this
contradicts the Hopf Theorem unless $v_{\mu_1}\equiv 0$. If
$v_{\mu_1}\equiv 0$, then the proof of Theorem \ref{t 1.1} is
complete. So if $u\not\equiv 0$, then we must have $\mu_1\leq -\frac
{\sqrt{2}}{2}$.
\end{proof}

Now, let $\mu\in [\mu_1,0)$ fixed, then the same procedure as above shows that we move can the plane $T_\lambda$ from $\lb=0$ to
$\lambda=\mu$, in particular, we have $v_\mu(x)\leq v_\mu(x^\mu)$ for $x\in \Sigma_\mu$ and $\mu\in [\mu_1,0)$. Based on the information that
$v_{\mu_1}(x)\leq v_{\mu_1}(x^{\mu_1})$ for $x\in \Sigma_{\mu_1}$, we can move the plane $\mu$ from $\mu_1=-\frac{\sqrt 2}{2}$ to $\mu_2=-\frac{\sqrt 2+3}{4}$,
then for $\mu\in [\mu_2,\mu_1)$, we can move the plane $T_\lambda$ from $\mu_1$ to $\mu$. In particular, we have $v_\mu(x)\leq v_\mu(x^\mu)$ for $x\in \Sigma_\mu$ and $\mu\in [\mu_2,\mu_1)$. Continuing the above procedure, in the $(k+1)$th step, we can show that the parameter $\mu$ can be moved from $\mu_k$ to $\mu_{k+1}=\frac{\mu_k-\sqrt{\mu_k^2+2}}{2}$. A direct calculation shows that $\mu_k\to -\infty$ as $k\to \infty$. Hence, we have proved
\begin{proposition}\label{t 2.7}
For any $\mu<0$, we have $v_\mu(x)\leq v_\mu(x^\mu)$ for $x\in \Sigma_{\mu}$.
\end{proposition}

\begin{proposition}\label{t 2.8}
For fixed $(x_2,x_3,\cdots,x_N)\in \mathbb R^{N-1}$,
$u(x_1,x_2,x_3,\cdots,x_N)$ is nondecreasing as $x_1$ decreases.
\end{proposition}
\begin{proof}
By Corollary \ref{t 2.22} and Proposition \ref{t 2.7}, we have $v_\mu\leq
v_\mu^\mu$ for any $\mu\in \mathbb R$. In particular, we have
$$\frac
1{|x-p_\mu|^{N-2}}v(p_\mu+\frac{x-p_\mu}{|x-p_\mu|^2})\leq \frac
1{|x_\mu-p_\mu|^{N-2}}v(p_\mu+\frac{x^\mu-p_\mu}{|x^\mu-p_\mu|^2}).$$
Since $|x-p_\mu|=|x^\mu-p_\mu|$ and $x_1^\mu-\mu=\mu-x_1$, then we
have $$u(x_1,x_2,x_3,\cdots,x_N)\leq
u(2\mu-x_1,x_2,x_3,\cdots,x_N)$$ for $x_1\geq \mu$. Because $\mu$ is
arbitrary, then we conclude that $u$ is nondecreasing as $x_1$
decreases.

\end{proof}

\begin{proof}[Proof of Theorem \ref{t 1.1}:] To prove Theorem \ref{t 1.1},
we first start the moving plane method in the $x_2$ direction. For this
purpose, we define the Kelvin transformation $v_\mu$ of $u$ at
$p_\mu=(0,\mu,0,\cdots,0)$ as
$$
v_\mu(x)=\frac{1}{|x-p_\mu|^{N-2}}u(p_\mu+\frac{x-p_\mu}{|x-p_\mu|^{2}}).
$$
Moreover, we define $\Sigma_\lb=\{x\in \R\mid x_2>\lb\},
T_\lb=\{x\in \R\mid x_2=\lb\}$ and $x^\lb$ be the reflection of $x$
with respect to $T_\lb$ as before.

If we denote $v_\mu^\lb=v_\mu(x^\lb)$, then by the same reason as
Lemma \ref{t 2.1} and Lemma \ref{t 2.4}, we have the following
inequality
\begin{equation}\label{2.18}
\begin{split}
&\int_{\Sigma_\lambda}|\nabla (v-v^\lambda)^+|^2\,dx\\&\leq
C_\lambda[(\int_{A_\mu^\lambda}\frac
1{|x-p_\mu|^{2N}}\,dx)^{\frac2N}+(\int_{B_\mu^\lambda}\frac
1{|x-p_\mu|^{2(N-1)}}\,dx')^{\frac
1{N-1}}]\cdot(\int_{\Sigma_\lambda}|\nabla(v-v^\lambda)^+|^2\,dx),
\end{split}
\end{equation}
where $C_\lb$ is a nonincreasing constant, $A_\mu^\lb=\{x\in
\Sigma_\lb|v_\mu(x)>v_\mu^\lb\}$ and \\$B_\mu^\lb=\{x\in
\partial \Sigma_\lb|v_\mu(x)>v_\mu^\lb, x_N=0\}$. With the above
inequality, we can start the moving plane method for $\lb$ large
enough. Now, move the plane $T_\lb$ from the right to the left and
suppose the process stops at some $\lb_1$. If $\lb_1>\mu$, then we
can prove that $v_\mu$ is symmetric with respect to $T_{\lb_1}$ as
before. This means that $v_\mu$ is regular at $p_\mu$. In particular, we
have $u\in L^{\frac{2N}{N-2}}(\R)$. Otherwise, we must have
$\lb_1\leq \mu$, which implies $v_\mu\leq v_\mu^\mu$. At the same
time, we can start moving plane method from the left. Similarly, if
this process stops at some $\lb_1<\mu$, then $u\in
L^{\frac{2N}{N-2}}(\R)$, otherwise, we have $\lb_1=\mu$ and
$v_\mu^\mu\leq v_\mu$.

We have proved the following fact: if the moving plane procedure
stops at some $\lb_1\neq \mu$, then $u\in L^{\frac{2N}{N-2}}(\R)$,
otherwise, we must have $v_\mu\equiv v_\mu^\mu$ and hence $u\equiv
u^\mu$.

With the above preparations, we can prove Theorem \ref{t 1.1} now.
We distinguish two cases.

Case 1: During the moving plane method process in the
$x_2,\cdots,x_{N-1}$ directions, there exists some $\mu$ and some
direction, such that $\lb_1\neq \mu$, i.e., the moving plane method
stops at some $\lb_1\neq \mu$. Then we conclude that $u$ is regular
at infinity and hence $u\in L^{\frac{2N}{N-2}}(\R)$. This is
impossible unless $u\equiv 0$ since $u$ is nondecreasing as $x_1$
decreases by Proposition \ref{t 2.8}.

Case 2: For all $\mu$ and all directions $x_2,\cdots,x_{N-1}$, the
moving plane method stops at $\lb_1=\mu$. In this case, the function
$u$ is independent of $x_2,\cdots,x_{N-1}$ since $\mu$ is arbitrary.
So we have $u=u(x_1,x_N)$. Moreover, Proposition \ref{t 2.8} implies
that $u(x_1,x_n)$ is nondecreasing as $x_1$ decreases. On the other
hand, since $u$ is bounded, then the limit $w(x_N)=\lim_{x_1\to
-\infty} u(x_1,x_N)$ exists and $w(x_N)$ satisfies
\begin{equation}\label{2.19}
\left\{
  \begin{array}{ll}
  \displaystyle
w^{\prime\prime}(x_N)=-f(w)    &{\rm for}\quad x_N>0,    \\
\displaystyle
\\ \frac{\partial w}{\partial x_N}=-g(u) &{\rm for}\quad x_N=0.
\end{array}
\right.
\end{equation}
If $w\not\equiv c$ with $f(c)=g(c)=0$, then the first equation
implies that $w$ is a concave function, while the second equation
implies $w'(0)<0$. Hence there exists $r_0>0$ such that $w(x_N)<0$
for $x_N>r_0$, this contradicts that $u$ is a nonnegative solution to
problem \eqref{1.1}.

Finally, since $w\equiv c$ with $f(c)=g(c)=0$ and $f,g$ are
nondecreasing, then equation \eqref{1.1} turns to be
$$
\left\{
  \begin{array}{ll}
  \displaystyle
-\Delta u=0    &{\rm in}\quad \R,    \\
\displaystyle
\\ \frac{\partial u}{\partial \nu}=0 &{\rm on}\quad \partial \R.
\end{array}
\right.
$$
Since $u$ is bounded, then the classical Liouville theorem implies
$u\equiv c$ with $f(c)=g(c)=0$. This finishes the proof of Theorem
\ref{t 1.1}.

\end{proof}

\section{\bf Proof of Theorem \ref{t 1.2}}
The sprit of the proof of Theorem \ref{t 1.2} is the same as the
proof of Theorem \ref{t 1.1}. We only give the difference between
the two problems.

As before, for any $\mu\in \mathbb R $, we define the Kelvin
transformation $v_\mu$ of $u$ at $p_\mu=(\mu,0,...,0)$ as
$$
v_\mu(x)=\frac
1{|x-p_\mu|^{N-2}}u(p_\mu+\frac{x-p_\mu}{|x-p_\mu|^2}),
$$
then a direct calculations shows that $v_\mu$ satisfies
\begin{equation}\label{3.1}
 \left\{
\begin{array}{ll}
-\Delta v_\mu=\frac{f(|x-p_\mu|^{N-2}v_\mu(x))}{[|x-p_\mu|^{N-2}v_\mu(x)]^{\frac{N+2}{N-2}}}v_\mu(x)^{\frac{N+2}{N-2}}, & \ x\in \R,\\
\\ \frac{\partial v_\mu}{\partial \nu}=\frac{g(|x-p_\mu|^{N-2}v_\mu(x))}{[|x-p_\mu|^{N-2}v_\mu(x)]^{\frac{N}{N-2}}}v_\mu(x)^{\frac{N}{N-2}}, & \ x\in \{x\in \partial \R:\mu-\frac
1\mu<x_1<\mu\},\\
\\ v_\mu=0,& \ x\in \{x\in \partial \R :x_1>\mu\ {\rm or}\ x_1<\mu-\frac
1\mu\}
\end{array}
\right.
\end{equation}
for $\mu\geq 0$, where we interpret $\frac 1\mu=+\infty$ for
$\mu=0$. However, for $\mu<0$, $v_\mu$ satisfies
\begin{equation}\label{3.2}
 \left\{
\begin{array}{ll}
-\Delta v_\mu=\frac{f(|x-p_\mu|^{N-2}v_\mu(x))}{[|x-p_\mu|^{N-2}v_\mu(x)]^{\frac{N+2}{N-2}}}v_\mu(x)^{\frac{N+2}{N-2}}, & \ x\in \R,\\
\\ \frac{\partial v_\mu}{\partial \nu}=\frac{g(|x-p_\mu|^{N-2}v_\mu(x))}{[|x-p_\mu|^{N-2}v_\mu(x)]^{\frac{N}{N-2}}}v_\mu(x)^{\frac{N}{N-2}}, & \ x\in\{x\in \partial \R:x_1<\mu\ {\rm or}\ x_1>\mu-\frac
1\mu\}, \\
\\ v_\mu=0,& \ x\in \{x\in \partial \R:\mu<x_1<\mu-\frac
1\mu\}.
\end{array}
\right.
\end{equation}

Moreover, we infer from the definitions of $h$ and $k$ that $v_\mu$
satisfies
\begin{equation}\label{3.3}
 \left\{
\begin{array}{ll}
-\Delta v_\mu=h(|x-p_\mu|^{N-2}v_\mu(x))v_\mu(x)^{\frac{N+2}{N-2}}, & \ x\in \R,\\
\\ \frac{\partial v_\mu}{\partial \nu}=k(|x-p_\mu|^{N-2}v_\mu(x))v_\mu(x)^{\frac{N}{N-2}}, & \ x\in \{x\in \partial \R:\mu-\frac
1\mu<x_1<\mu\},\\
\\ v_\mu=0,& \ x\in \{x\in \partial \R:x_1>\mu\ {\rm or}\ x_1<\mu-\frac
1\mu\}
\end{array}
\right.
\end{equation}
for $\mu\geq 0$ and $v_\mu$ satisfies
\begin{equation}\label{3.4}
 \left\{
\begin{array}{ll}
-\Delta v_\mu=h(|x-p_\mu|^{N-2}v_\mu(x))v_\mu(x)^{\frac{N+2}{N-2}}, & \ x\in \R,\\
\\ \frac{\partial v_\mu}{\partial \nu}=k(|x-p_\mu|^{N-2}v_\mu(x))v_\mu(x)^{\frac{N}{N-2}}, & \ x\in\{x\in \partial \R:x_1<\mu\ {\rm or}\ x_1>\mu-\frac
1\mu\}, \\
\\ v_\mu=0,& \ x\in \{x\in \partial \R:\mu<x_1<\mu-\frac
1\mu\}
\end{array}
\right.
\end{equation}
for $\mu<0$.

As before, we first start the moving plane method in the $x_1$
direction for $\mu\geq 0$. For problem \eqref{3.4}, we have the same
inequality as Lemma \ref{t 1.2}.
\begin{lemma}\label{t 3.1}
For any fixed $\lambda>\mu\geq 0$, the functions $v_\mu$ and
$(v_\mu-v_\mu^\lambda)^+$ belong to $ L^{2^*}(\Sigma_\lambda)\cap
L^\infty(\Sigma_\lambda)$ with $2^*=\frac{2N}{N-2}$. Further more,
if we denote $A_\mu^\lambda=\{x\in \Sigma_\lambda|
v_\mu>v_\mu^\lambda\}$ and $B_\mu^\lambda=\{x\in \partial
\Sigma_\lambda^1|v_\mu(x)>v_\mu^\lambda(x)\}$, then there exists
$C_\lambda>0$, which is nonincreasing in $\lambda$, such that
\begin{equation}\label{3.7}
\int_{\Sigma_\lambda}|\nabla (v_\mu-v_\mu^\lambda)^+|^2\,dx\leq
C_\lambda(\int_{A_\mu^\lambda}\frac
1{|x-p_\mu|^{2N}}\,dx)^{\frac2N}\cdot(\int_{\Sigma_\lambda}|\nabla(v_\mu-v_\mu^\lambda)^+|^2\,dx).
\end{equation}
\end{lemma}
\begin{proof}
We observe from the boundary value condition of $v_\mu$ that
$$
\frac{\partial(v_\mu-v_\mu^\lb)}{\partial \nu}\varphi(x)=0
$$
for any $x\in \partial \Sigma_\lb^1$. Then the rest of the proof is
the same as the proof of Lemma \ref{t 2.1}.
\end{proof}

With the above inequality, we can start the moving plane procedure
now. Similar to the proof of Theorem \ref{t 1.1}, we have
\begin{proposition}\label{t 3.2}
For any $\mu\geq 0$, we have $v_\mu(x)\leq v_\mu^\mu(x)$ for $x\in \Sigma_\mu$.
\end{proposition}

Next, we consider the case $\mu<0$. Similar to the result for
problem \eqref{1.1}, we have the following result.
\begin{lemma}\label{t 3.3}
Suppose that $\mu<0$, then for any fixed $\mu<\lambda< \mu-\frac
1{2\mu}$, the functions $v_\mu$ and $(v_\mu-v_\mu^\lambda)^+$ belong
to $ L^{2^*}(\Sigma_\lambda)\cap L^\infty(\Sigma_\lambda)$ with
$2^*=\frac{2N}{N-2}$. Further more, if we denote
$A_\mu^\lambda=\{x\in \Sigma_\lambda| v_\mu>v_\mu^\lambda\}$ and
$B_\mu^\lambda=\{x\in \partial
\Sigma_\lambda^1|v_\mu(x)>v_\mu^\lambda(x)\}$, then there exists
$C_\lambda>0$, which is nonincreasing in $\lambda$, such that
\begin{equation}\label{3.8}
\begin{split}
&\int_{\Sigma_\lambda}|\nabla (v_\mu-v_\mu^\lambda)^+|^2\,dx\\&\leq
C_\lambda[(\int_{A_\mu^\lambda}\frac
1{|x-p_\mu|^{2N}}\,dx)^{\frac2N}+(\int_{B_\mu^\lambda}\frac
1{|x-p_\mu|^{2(N-1)}}\,dx')^{\frac
1{N-1}}]\cdot(\int_{\Sigma_\lambda}|\nabla(v_\mu-v_\mu^\lambda)^+|^2\,dx).
\end{split}
\end{equation}
\end{lemma}
\begin{proof}
We only need to note that
$$
\int_{\partial \Sigma_\lb}\frac{\partial(v_\mu-v_\mu^\lb)}{\partial
\nu}\varphi(x)\,dx'\leq
\int_{B_\mu^\lb}[k(|x-p_\mu|^{N-2}v_\mu(x))v_\mu^{\frac{N}{N-2}}
-k(|x^\lb-p_\mu|^{N-2}v_\mu^\lambda)(v_\mu^\lambda)^{\frac{N}{N-2}}]\varphi\,dx'.
$$
The rest of the proof is the same the proof of Lemma \ref{t 2.4}, we
omit it.

\end{proof}

With the above inequality, we can start the moving plane procedure for $\mu<0$. First, we can move $\mu$ from $\mu_0=0$ to $\mu_1=-\frac{\sqrt2}2$, then for
$\mu\in [\mu_1,0)$, we can move the plane $T_\lambda$ from $\lambda=0$ to $\lambda=\mu$. That is, we have
$v_\mu(x)\leq v_{\mu}(x^\mu)$ for any $\mu\in [\mu_1,0)$ and $x\in \Sigma_\mu$. Next, we can move the plane $\mu$ from $\mu_1$ to $\mu_2=-\frac{3+\sqrt 2}4$. Similarly, for any $\mu\in [\mu_2,\mu_1)$, we can move the plane $T_\lambda$ from $\lambda=\mu_1$ to $\lambda=\mu$. That is, we have
$v_\mu(x)\leq v_{\mu}(x^\mu)$ for any $\mu\in [\mu_2,\mu_1)$ and $x\in \Sigma_\mu$. Continue the above procedure, in the $(k+1)$th step, the plane $\mu$ can be moved from $\mu_{k}$ from $\mu_{k+1}=\frac{\mu_k-\sqrt{\mu_k^2+2}}{2}$ and $T_\lb$ can be moved from $\mu_k$ to $\mu\in [\mu_{k+1},\mu_k)$. A direct calculation shows that $\mu_k\to -\infty$ as $k\to \infty$. So we have proved the following result.
\begin{proposition}\label{t 3.4}
For any $\mu< 0$, we have $v_\mu(x)\leq v_\mu^\mu(x)$ for $x\in \Sigma_\mu$.
\end{proposition}

As a direct consequence of Proposition \ref{t 3.2} and Proposition \ref{t 3.4},
we have the following result.
\begin{proposition}\label{t 3.5}
For fixed $(x_2,x_3,\cdots,x_N)\in \mathbb R^{N-1}$,
$u(x_1,x_2,x_3,\cdots,x_N)$ is nondecreasing as $x_1$ decreases.
\end{proposition}

With the above results, we are in the position of proving Theorem
\ref{t 1.2} now.
\begin{proof}[Proof of Theorem \ref{t 1.2}]
As the proof of Theorem \ref{t 1.1}, we first start the moving plane
method in the $x_2$ direction. For this purpose, we define the Kelvin
transformation $v_\mu$ of $u$ at $p_\mu=(0,\mu,0,\cdots,0)$ as
$$
v_\mu(x)=\frac{1}{|x-p_\mu|^{N-2}}u(p_\mu+\frac{x-p_\mu}{|x-p_\mu|^{2}}).
$$
and $\Sigma_\lb,T_\lb,x^\lb$ as before.

If we denote $v_\mu^\lb=v_\mu(x^\lb)$, then by the same reason as
Lemma \ref{t 3.1} and Lemma \ref{t 3.3}, we have the following
inequality
\begin{equation}\label{3.7}
\begin{split}
&\int_{\Sigma_\lambda}|\nabla (v-v^\lambda)^+|^2\,dx\\&\leq
C_\lambda[(\int_{A_\mu^\lambda}\frac
1{|x-p_\mu|^{2N}}\,dx)^{\frac2N}+(\int_{B_\mu^\lambda}\frac
1{|x-p_\mu|^{2(N-1)}}\,dx')^{\frac
1{N-1}}]\cdot(\int_{\Sigma_\lambda}|\nabla(v-v^\lambda)^+|^2\,dx)
\end{split}
\end{equation}
with $C_\lb,A_\lb$ and $B_\lb$ as before. With this inequality, we
can start the moving plane method for $\lb$ large enough. Now,
moving the plane $T_\lb$ from the right to the left and suppose that the
process stops at some $\lb_1$. If $\lb_1>\mu$, then we can prove
that $v_\mu$ is symmetric with respect to $T_{\lb_1}$ as before.
This means that $v_\mu$ is regular at $p_\mu$. In particular, we have
$u\in L^{\frac{2N}{N-2}}(\R)$. Otherwise, we must have $\lb_1\leq
\mu$, which implies $v_\mu\leq v_\mu^\mu$. At the same time, we can
start the moving plane method from the left. Similarly, if this process
stops at some $\lb_1<\mu$, then $u\in L^{\frac{2N}{N-2}}(\R)$,
otherwise, we have $\lb_1=\mu$ and $v_\mu^\mu\leq v_\mu$.

So during the process of moving plane in the $x_2,\cdots,x_{N-1}$
directions, two cases may occur.

Case 1: There exists some direction and some $\mu$, such that
$\lb_1\neq \mu$, i.e., the moving plane method stops at some
$\lb_1\neq \mu$. Then we conclude that $u$ is regular at infinity
and hence $u\in L^{\frac{2N}{N-2}}(\R)$. This is impossible unless
$u\equiv 0$ since $u$ is nondecreasing as $x_1$ decrease by
Proposition \ref{t 3.5}.

Case 2: For all $\mu$ and all directions $x_2,\cdots,x_{N-1}$, the
moving plane method stops at $\lb_1=\mu$. In this case, the function
$u$ is independent of $x_2,\cdots,x_{N-1}$ since $\mu$ is arbitrary.
So we have $u=u(x_1,x_N)$. Moreover, Proposition \ref{t 3.5} implies
that $u(x_1,x_n)$ is nondecreasing as $x_1$ decreases. On the other
hand, since $u$ is bounded, then the limit function
$w(x_N)=\lim_{x_1\to -\infty} u(x_1,x_N)$ exists and $w(x_N)$
satisfies
\begin{equation}\label{2.19}
\left\{
  \begin{array}{ll}
  \displaystyle
w^{\prime\prime}(x_N)=-f(w)    &{\rm for}\quad  x_N>0,    \\
\displaystyle
\\ \frac{\partial w}{\partial x_N}=-g(u) &{\rm for}\quad x_N=0.
\end{array}
\right.
\end{equation}
If $w\not\equiv c$ with $f(c)=g(c)=0$, then the first equation implies that $w$ is a
concave function, while the second equation implies $w'(0)<0$. Hence
there exists $r_0>0$ such that $w(x_N)<0$ for $x_N>r_0$, this
contradicts that $u$ is a nonnegative solution to problem \eqref{1.2}.
So we have $w\equiv c$ and the Dirichlet boundary value condition implies $u\equiv 0$. This finishes the
proof of Theorem \ref{t 1.2}.

\end{proof}

{\bf Acknowledgement}: This work is supported by NSFC, No.11101291.

\end{document}